\documentclass[12pt]{amsart}
\usepackage{amssymb}

\numberwithin{equation}{section}

\newtheorem{theorem}{Theorem}[section]
\newtheorem{defi}{Definition}[section]
\newtheorem{conj}{Conjecture}[section]

\newtheorem{corollary}{Corollary}[section]

\newtheorem*{theorema}{Theorem A}
\newtheorem*{theoremb}{Theorem B}
\newtheorem*{theoremc}{Theorem C}

\newcommand{\N}{\mathbb{N}}

\newcommand{\Q}{\mathbb{Q}}

\newcommand{\G}{{G}}
\newcommand{\A}{\mathcal{A}}
\newcommand{\B}{\mathcal{B}}
\newcommand{\C}{\mathcal{C}}
\newcommand{\R}{\mathcal{R}}
\newcommand{\F}{\mathcal{F}}
\newcommand{\g}{\mathcal{G}}

\newcommand{\el}{\mathcal{L}} 

\title[Decompositions of sets of integers]{On additive and multiplicative decompositions of sets of integers\\ with restricted prime factors, II.\\ (Smooth numbers and generalizations.)}

\author[K. Gy\H{o}ry, L. Hajdu and A. S\'ark\"ozy]{K. Gy\H{o}ry, L. Hajdu and A. S\'ark\"ozy}

\thanks{Research supported in part by the NKFIH grants K115479, K119528, K128088, and K130909, and by the projects EFOP-3.6.1-16-2016-00022 and EFOP-3.6.2-16-2017-00015 of the European Union, co-financed by the European Social Fund.}



\address{K. Gy\H{o}ry
\hfill\break\indent L. Hajdu
\hfill\break\indent University of Debrecen, Institute of Mathematics
\hfill\break\indent H-4002 Debrecen, P.O. Box 400.
\hfill\break\indent Hungary}

\email{gyory@science.unideb.hu}
\email{hajdul@science.unideb.hu}

\address{A. S\'ark\"ozy
\hfill\break\indent E\"otv\"os Lor\'and University, Institute of Mathematics
\hfill\break\indent H-1117 Budapest, P\'azm\'any P\'eter s\'et\'any 1/C
\hfill\break\indent Hungary}

\email{sarkozy@cs.elte.hu}

\begin{document}

\begin{abstract}
In part I of this paper we studied additive decomposability of the set $\F_y$ of th $y$-smooth numbers and the multiplicative decomposability of the shifted set $\g_y=\F_y+\{1\}$. In this paper, focusing on the case of 'large' functions $y$, we continue the study of these problems. Further, we also investigate a problem related to the m-decomposability of $k$-term sumsets, for arbitrary $k$.
\end{abstract}

\date{\today}

\maketitle

\section{Introduction}

First we recall some notation, definitions and results from part I of this paper \cite{ghs} which we all also need here.

$\A,\B,\C,\hdots$ denote (usually infinite) sets of non-negative integers, and their counting functions are denoted by $A(X),B(X),C(X),\hdots$ so that e.g.
$$
A(X)=|\{a:a\leq X,\ a\in\A\}|.
$$
The set of the positive integers is denoted by $\N$, and we write $\N\cup\{0\}=\N_0$. The set of rational numbers is denoted by $\Q$.

We will need
\begin{defi}
\label{defi1}
Let $\G$ be an {\rm additive} semigroup and $\A,\B,\C$ subsets of $\G$ with
\begin{equation}
\label{eq11}
|\B|\geq 2,\ \ \ |\C|\geq 2.
\end{equation}
If
\begin{equation}
\label{eq12}
\A=\B+\C\ (=\{b+c:b\in\B,\ c\in\C\})
\end{equation}
then \eqref{eq12} is called an {\rm additive decomposition} or briefly {\rm a-decomposition} of $\A$, while if a {\rm multiplication} is defined in $\G$ and \eqref{eq11} and
\begin{equation}
\label{eq13}
\A=\B\cdot\C\ (=\{bc:b\in\B,\ c\in\C\})
\end{equation}
hold then \eqref{eq13} is called a {\rm multiplicative decomposition} or briefly {\rm m-decomposition} of $\A$.
\end{defi}

\begin{defi}
\label{defi2}
A finite or infinite set $\A$ of non-negative integers is said to be {\rm a-reducible} if it has an {\rm additive decomposition}
$$
\A=\B+\C\ \ \ \text{with}\ \ \ |\B|\geq 2,\ |\C|\geq 2
$$
(where $\B\subset\N_0$, $\C\subset\N_0$). If there are no sets $\B,\C$ with these properties then $\A$ is said to be {\rm a-primitive} or {\rm a-irreducible}.
\end{defi}

\begin{defi}
\label{defi3}
Two sets $\A,\B$ of non-negative integers are said to be {\rm asymptotically equal} if there is a number $K$ such that $\A\cap[K,+\infty)=\B\cap [K,+\infty)$ and then we write $\A\sim\B$.
\end{defi}

\begin{defi}
\label{defi4}
An infinite set $\A$ of non-negative integers is said to be {\rm totally a-primitive} if every $\A'$ with $\A'\subset \N_0$, $\A'\sim\A$ is {\rm a-primitive}.
\end{defi}

The multiplicative analogs of Definitions \ref{defi2} and \ref{defi4} are:

\begin{defi}
\label{defi5}
If $\A$ is an infinite set of {\rm positive} integers then it is said to be {\rm m-reducible} if it has a {\rm multiplicative decomposition}
$$
\A=\B\cdot\C\ \ \ \text{with}\ \ \ |\B|\geq 2,\ |\C|\geq 2
$$
(where $\B\subset\N$, $\C\subset\N$). If there are no such sets $\B,\C$ then $\A$ is said to be {\rm m-primitive} or {\rm m-irreducible}.
\end{defi}

\begin{defi}
\label{defi6}
An infinite set $\A\subset\N$ is said to be {\rm totally m-primitive} if every $\A'\subset \N$ with $\A'\sim\A$ is {\rm m-primitive}.
\end{defi}

\begin{defi}
\label{defi7}
Denote the greatest prime factor of the positive integer $n$ by $p^+(n)$. Then $n$ is said to be smooth (or friable) if $p^+(n)$ is ''small'' in terms of $n$. More precisely, if $y=y(n)$ is a monotone increasing function on $\N$ assuming positive values and $n\in\N$ is such that $p^+(n)\leq y(n)$, then we say that $n$ is $y$-smooth, and we write $\F_y$ ($\F$ for ''friable'') for the set of all $y$-smooth positive integers.
\end{defi}

Starting out from a conjecture of the third author \cite{sa} and a related partial result of Elsholtz and Harper \cite{eh}, in \cite{ghs} we proved the following two theorems:

\begin{theorema}
\label{thma}
If $y(n)$ is an increasing function with $y(n)\to\infty$ and
\begin{equation}
\label{eq15}
y(n)<2^{-32}\log n\ \ \ \text{for large}\ n,
\end{equation}
then the set $\F_y$ is totally {\rm a-primitive}.
\end{theorema}

(If $y(n)$ is increasing then the set $\F_y$ is m-reducible since $\F_y=\F_y\cdot\F_y$, and we also have $\F_y\sim\F_y\cdot\{1,2\}$, thus if we want to prove an {\sl m-primitivity} theorem involving $\F_y$ then we have to switch from $\F_y$ to the shifted set
\begin{equation}
\label{eq16}
\g_y:=\F_y+\{1\}.
\end{equation}
See also \cite{el}.)

\begin{theoremb}
\label{thmb}
If $y(n)$ is defined as in Theorem \ref{thm1}, then the set $\g_y$ is totally {\sl m-primitive}.
\end{theoremb}

Here our goal is to prove further related results. First we will prove a theorem in the direction opposite to the one in Theorem A. Indeed, we will show that if $y(n)$ grows faster than $n/2$, then $\F_y$ is {\it not} totally a-primitive.

\begin{theorem}
\label{thm1}
Let $y(n)$ be any monotone increasing function on $\N$ with
$$
\frac{n}{2}<y(n)<n\ \ \ \text{for all}\ n\in\N.
$$
Then $\F_y$ is {\rm not} totally a-primitive. In particular, in this case the set
$$
\F_y\cap [9,+\infty)
$$
is a-reducible, namely, we have
$$
\F_y\cap [9,+\infty) = \A+\B
$$
with
$$
\A=\{n\in\N:\text{none of}\ n,n+1,n+3,n+5\ \text{is prime}\},\ \B=\{0,1,3,5\}.
$$
\end{theorem}

Next we will show that under a standard conjecture, the decomposition in Theorem \ref{thm1} is best possible in the sense that no such decomposition is possible with $2\leq |\B|\leq 3$. For this, we need to formulate the so-called prime $k$-tuple conjecture. A finite set $A$ of integers is called admissible, if for any prime $p$, no subset of $A$ forms a complete residue system modulo $p$.

\begin{conj}[The prime $k$-tuple conjecture]
\label{c1}
Let $\{a_1,\dots,a_k\}$ be an admissible set of integers. Then there exist infinitely many positive integers $n$ such that $n+a_1,\dots,n+a_k$ are all primes.
\end{conj}

\noindent
{\bf Remark.} By a recent, deep result of Maynard \cite{m} we know that for each $k$, the above conjecture holds for a positive proportion of admissible $k$-tuples. We also mention that if the prime $k$-tuple conjecture is true, then there exist infinitely many $n$ such that $n+a_1,\dots,n+a_k$ are {\sl consecutive} primes (see e.g. the proof of Theorem 2.4 of \cite{hs}).

\begin{theorem}
\label{prop1}
Define $y(n)$ as in Theorem \ref{thm1} and suppose that the prime $k$-tuple conjecture is true for $k=2,3$. Then for any $\C\subset\N$ with $\C\sim \F_y$ there is no decomposition of the form
$$
\C=\A+\B
$$
with
$$
2\leq |\B|\leq 3.
$$
\end{theorem}

We propose the following problem, which is a shifted, multiplicative analogue of the question studied in Theorems \ref{thm1} and \ref{prop1}.

\noindent{\bf Problem.} With the same $y=y(n)$ as in Theorem \ref{thm1}, write
$$
\g_y=\F_y+\{1\}=\{m+1:m\in\F_y\}.
$$
Is the set $\g_y$ totally m-primitive?

Towards the above problem, we prove that no appropriate decomposition is possible with $|\B|<+\infty$.

\begin{theorem}
\label{prop2}
Let $y(n)$ be as in Theorem \ref{thm1}. Then for any $\C\subset\N$ with $\C\sim \g_y$ there is no decomposition of the form
$$
\C=\A\cdot\B
$$
with
$$
|\B|<+\infty.
$$
\end{theorem}

\vskip.2cm

Let
$$
\Gamma:=\{n_1,\dots,n_s\}
$$
be a set of pairwise coprime positive integers $>1$, and let $\{\Gamma\}$ be the multiplicative semigroup generated by $\Gamma$, with $1\in\{\Gamma\}$. If in particular, $n_1,\dots,n_s$ are distinct primes, then we use the notation $\Gamma=S$, and $\{\Gamma\}=\{S\}$ is just the set of positive integers composed of the primes from $S$.

The next theorem shows that if $\Gamma$ is finite, then the sets of $k$-term and at most $k$-term sums of pairwise coprime elements of $\{\Gamma\}$ are totally m-primitive. For the precise formulation of the statement, write $H_1:=\{\Gamma\}$, and for $k\geq 2$ set
$$
H_k:=\{u_1+\dots+u_k:\ u_i\in\{\Gamma\},\ \gcd(u_i,u_j)=1\ \text{for}\ 1\leq i<j\leq k\}
$$
and
$$
H_{\leq k}:=\bigcup\limits_{\ell=1}^k H_\ell.
$$

\begin{theorem}
\label{thm2}
Let $k\geq 2$. Then both $H_k$ and $H_{\leq k}$ are totally m-primitive, apart from the only exception exception of the case $\Gamma=\{2\}$ and $k=3$, when we have
$$
H_{\leq 3}=\{1,2\}\cdot\{2^\beta,2^{\beta}+1:\beta\geq 0\}.
$$
\end{theorem}

\noindent{\bf Remark.} As we have
$$
\{\Gamma\}=\{\Gamma\}\cdot \{\Gamma\},
$$
the assumption $k\geq 2$ is clearly necessary. Further, the coprimality assumption in the definition of $H_k$ cannot be dropped. Indeed,
letting
$$
H_k^*:=\{u_1+\dots+u_k:\ u_i\in\{\Gamma\}\ \text{for}\ 1\leq i\leq k\}
$$
and
$$
H_{\leq k}^*:=\bigcup\limits_{\ell=1}^k H_\ell^*
$$
we clearly have
$$
H_k^*=\{\Gamma\}\cdot H_k^*\ \ \ \text{and}\ \ \ H_{\leq k}^*=\{\Gamma\}\cdot H_{\leq k}^*.
$$

\section{Proof of Theorem \ref{thm1}}

By the choice of $y(n)$ we see that $\F_y$ is the set of all composite integers. Put
$$
\C=\F_y\cap [9,+\infty).
$$
We show that by the definition of $\A$ and $\B$ as in the theorem, we have
$$
\C=\A+\B.
$$
To see this, first observe that by the assumptions on $\A$ and $\B$, all the elements of $\A+\B$ are composite. So we only need to check that all composite numbers $n$ with $n\geq 9$ belong to $\A+\B$. If $n$ is an odd composite number, then by $n\in\A$ we have
\begin{equation}
\label{eq00}
n\in(\A+\B).
\end{equation}
So assume that $n$ is an even composite number with $n\geq 10$. Then one of $n-1,n-3,n-5$ is {\sl not} a prime. As this number is clearly in $\A$, we have \eqref{eq00} again and our claim follows.
$\qed$

\section{Proof of Theorem \ref{prop1}}

Let $\C\subset\N$ with $\C\sim\F_y$. Then, as we noted in the proof of Theorem \ref{thm1}, with some positive integer $n_0$ we have
$$
\C\cap [n_0,+\infty)=\{n\in\N:n\geq n_0\ \text{and}\ n\  \text{is composite}\}.
$$
We handle the cases $k=2$ and $3$ separately.

Let first $k=2$, that is assume that contrary to the assertion of the theorem the set $\C$ can be represented as
\begin{equation}
\label{ossz}
\C=\A+\B
\end{equation}
with $|\B|=2$. Set $B=\{b_1,b_2\}$. Clearly, without loss of generality we may assume that $b_1<b_2$ and also that $b_1=0$. Indeed, the first assumption is trivial, and the second one can be made since \eqref{ossz} implies that
$$
\C=\A^*+\{0,b_2-b_1\}
$$
with
$$
\A^*=\A+\{b_1\}=\{a+b_1:a\in\A\}.
$$
As the set $\{-b_2,b_2\}$ is admissible, by our assumption on the validity of Conjecture \ref{c1} we get that there exist infinitely many integers $n$ such that $n-b_2$ and $n+b_2$ are both primes. In view of the Remark after Conjecture \ref{c1}, we may assume that these primes are consecutive, that is, in particular, $n$ is composite. Observe that then, assuming that $n\geq n_0+b_2$, we have $n-b_2\notin\A$ and $n\notin\A$. Indeed, otherwise by the primality of $n-b_2$ and $n+b_2$, respectively, we get a contradiction: in case of $n-b_2\in\A$ we have $n-b_2\in\C$, while $n+b_2\in\A$ implies that $n+b_2\in\C$. But then we get $n\notin\C$, which is a contradiction.

Let now $k=3$, that is assume that we have \eqref{ossz} with some $\B$ with $|\B|=3$. Write $\B=\{b_1,b_2,b_3\}$. As in the case $k=2$, we may assume that $0=b_1<b_2<b_3$. Now we construct an admissible triple related to $\B$. If $b_2$ and $b_3$ are of the same parity, then either
$$
t_1=\{-b_3,-b_2,b_3\}
$$
or
$$
t_2=\{-b_3,-b_2,b_2\}
$$
is admissible, according as $3\mid b_3$ or $3\nmid b_3$. Further, if $b_2$ is odd and $b_3$ is even, then either
$$
t_3=\{-b_3+b_2,b_3-b_2,b_2\}
$$
or
$$
t_4=\{-b_3+b_2,-b_2,b_2\}
$$
is admissible, according as $b_2\equiv b_3\pmod{3}$ or $b_2\not\equiv b_3\pmod{3}$. Finally, if $b_2$ is even and $b_3$ is odd, then either
$$
t_5=\{-b_3+b_2,b_3-b_2,b_3\}
$$
or
$$
t_6=\{-b_3,b_3-b_2,b_3\}
$$
is admissible, according as $b_2\equiv b_3\pmod{3}$ or $b_2\not\equiv b_3\pmod{3}$. Let $1\leq i\leq 6$ such that $t_i$ is admissible, and write $t_i=\{u_1,u_2,u_3\}$. According to Conjecture \ref{c1} (see also the Remark after it) we get that there exists an $n$ with $n\geq n_0+b_3$ such that $n$ is composite, but
$$
n+u_1,\ \ \ n+u_2,\ \ \ n+u_3
$$
are all primes $\geq n_0$. However, then a simple check shows that for any value of $i$, we have that none of $n-b_3,n-b_2,n$ is in $\A$, since otherwise $\C$ would contain a prime $\geq n_0$. However, then we get $n\notin\C$. This is a contradiction, and our claim follows.
$\qed$

\section{Proof of Theorem \ref{prop2}}

Let $\C\subset\N$ with $\C\sim\g_y$. Then with some positive integer $n_0$ we have
$$
\C\cap [n_0,+\infty)=\{n+1:n\geq n_0-1\ \text{and}\ n\  \text{is composite}\}.
$$
Assume to the contrary that we can write
\begin{equation}
\label{szor}
\C=\A\cdot\B
\end{equation}
with $|\B|<+\infty$. Put $B=\{b_1,\dots,b_\ell\}$ with $\ell\geq 2$ and $1\leq b_1<b_2<\dots<b_\ell$.

Assume first that $1\notin\B$ (that is, $b_1>1$). Let $n$ be an arbitrary (composite) multiple of the product $b_1\dots b_\ell$ such that $n\geq n_0$. Then we immediately see that $n+1$ is not divisible by any $b_i$ $(i=1,\dots,\ell)$, which shows that $n+1\notin\C$. However, this is a contradiction, and our claim follows in this case.

Suppose now that $1\in\B$ (that is, $b_1=1$). For each of $i=2,\dots,\ell$ choose a prime divisor $p_i$ of $b_i$, with the convention that $p_i=4$ if $b_i$ is a power of $2$, and let $P$ be the set of these primes. Observe that $P$ is non-empty. Take two distinct primes $q_1,q_2$ not belonging to $P$, and consider the following system of linear congruences:
\begin{align*}
x\equiv 0&\pmod{q_i}\ \ \ \text{for}\ i=1,2,\\
x\equiv 1&\pmod{p}\ \ \ \text{if}\ p\in P,\ p\mid b_2-1,\\
x\equiv 0&\pmod{p}\ \ \ \text{if}\ p\in P,\ p\nmid b_2-1.
\end{align*}
Let $x_0$ be an arbitrary positive solution to the above system. Put
$$
N:=q_1q_2\prod\limits_{p\in P} p
$$
and consider the arithmetic progression
\begin{equation}
\label{ap}
(b_2N)t+(b_2(x_0+1)-1)
\end{equation}
in $t\geq 0$. Observe that here we have $\gcd(b_2N,b_2(x_0+1)-1)=1$. Indeed, $\gcd(b_2,b_2(x_0+1)-1)=1$ trivially holds, and as $b_2(x_0+1)-1=b_2x_0+b_2-1$, the relation $\gcd(N,b_2(x_0+1)-1)=1$ follows from the definition of $x_0$. Thus by Dirichlet's theorem there exist infinitely many integers $t$ such that \eqref{ap} is a prime. Let $t$ be such an integer with $t>n_0$, and put
$$
n:=tN+x_0.
$$
We claim that $n$ is composite with $n>n_0$, but $n+1\notin\C$. This will clearly imply the statement. It is obvious that $n>n_0$, and as $q_1q_2\mid N$ and $q_1q_2\mid x_0$, we also have that $n$ is composite. Further, we have that $n+1\notin\A$. Indeed, otherwise we would also have $b_2(n+1)\in\C$, that is, $b_2(n+1)-1$ should be composite - however,
$$
b_2(n+1)-1=(b_2N)t+(b_2(x_0+1)-1)
$$
is a prime. (The importance of this fact is that we cannot have $n+1\in\C$ by the relation $n+1=(n+1)\cdot 1$ with $n+1\in\A$ and $1\in\B$.) Further, since $n+1\equiv 1,2\pmod{p}$ for $p\in P$ as $p_i\geq 3$ and $p_i\mid b_i$ we have $b_i\nmid n+1$ for $i=3,\dots,\ell$. We need to check the case $i=2$ separately. If $b_2>2$, then we have $p_2\geq 3$ and $p_2\mid b_2$, and we have $b_2\nmid n+1$ again. On the other hand, if $b_2=2$ then as $b_2-1=1$ and $p_2=4$, we have $4\mid n$, so $b_2\nmid n+1$ once again. So in any case, $b_i\nmid n+1$ $(i=2,\dots,\ell)$. Hence $n+1$ cannot be of the form $ab_i$ with $a\in\A$ and $i=1,\dots,\ell$. Thus our claim follows also in this case.
$\qed$

\section{Proof of Theorem \ref{thm2}}

The proof of Theorem \ref{thm2} is based on the following deep theorem. Recall that $\{\Gamma\}$ denotes the multiplicative semigroup generated by $\Gamma$. Consider the equation
\begin{equation}
\label{k2}
a_1x_1+\dots+a_mx_m=0\ \ \ \text{in}\ x_1,\dots,x_m\in \{\Gamma\},
\end{equation}
where $a_1,\dots,a_m$ are non-zero elements of $\Q$. If $m\geq 3$, a solution of \eqref{k2} is called {\it non-degenerate} if the left hand side of \eqref{k2} has no vanishing subsums. Two solutions $x_1,\dots,x_m$ and $x_1',\dots,x_m'$ are {\it proportional} if
$$
(x_1',\dots,x_m')=\lambda(x_1,\dots,x_m)
$$
with some $\lambda\in \{\Gamma\}\setminus\{1\}$.

\begin{theoremc}
\label{thmc}
Equation \eqref{k2} has only finitely many non-proportional, non-degenerate solutions.
\end{theoremc}

This theorem was proved independently by van der Poorten and Schlickewei \cite{ps1} and Evertse \cite{ev} in a more general form. Later Evertse and Gy\H{o}ry \cite{egy1} showed that the number of non-proportional, non-degenerate solutions of \eqref{k2} can be estimated from above by a constant which depends only on $\Gamma$. For related results, see the paper \cite{ps2} and the book \cite{egy2}.

We shall use the following consequence of Theorem C.

\begin{corollary}
\label{cork}
There exists a finite set $\el$ such that if $x_1,\dots,x_\ell$ are pairwise coprime elements of $\{\Gamma\}$, $y_1,\dots,y_h$ are also pairwise coprime elements of $\{\Gamma\}$ such that $\ell,h\leq k$, $\ell+h\geq 3$ and
\begin{equation}
\label{k3}
\varepsilon(x_1+\dots+x_\ell)-\eta(y_1+\dots+y_h)=0
\end{equation}
with some $\varepsilon,\eta\in\{\Gamma\}$ and without vanishing subsum on the left hand side, then
$$
x_1,\dots,x_\ell,\ y_1,\dots,y_h\in\el.
$$
Further, $\el$ is independent of $\varepsilon,\eta$.
\end{corollary}

\begin{proof} Without loss of generality we may assume that $\ell\geq 2$. Then Theorem C implies that
$$
(\varepsilon x_1,\dots,\varepsilon x_\ell)=\nu(z_1,\dots,z_\ell),
$$
where $\nu,z_1,\dots,z_\ell\in\{\Gamma\}$, and $z_1,\dots,z_\ell$ belong to a finite set. Hence, as
$$
(x_1,\dots,x_\ell)=\nu^*(z_1,\dots,z_\ell)
$$
with $\nu^*=\nu/\varepsilon$, in view of that $x_1,\dots,x_\ell\in\{\Gamma\}$ are pairwise coprime, we conclude that $x_1,\dots,x_\ell$ belong to a finite set (which is independent of $\varepsilon,\eta$). If we have $h=1$, then expressing $y_1$ from \eqref{k3}, the statement immediately follows. On the other hand, if $h\geq 2$, then applying the above argument for $(\eta y_1,\dots,\eta y_h)$ in place of $(\varepsilon x_1,\dots,\varepsilon x_\ell)$, the statement also follows.
\end{proof}

Now we can prove our Theorem \ref{thm2}. Our argument will give the proof of our statement concerning both $H_k$ and $H_{\leq k}$. First note that there is a constant $C_1$ such that if in $H_k$ (resp. in $H_{\leq k}$) we have
$$
u_1+\dots+u_t>C_1
$$
with $t=k$ (resp. with $2\leq t\leq k$) and $\gcd(u_i,u_j)=1$ for $1\leq i<j\leq t$, then this sum is not contained in $\{\Gamma\}$. This is an immediate consequence of Theorem C.

Assume that contrary that contrary to the statement of the theorem for some $\R$ which is asymptotically equal to one of $H_k$ and $H_{\leq k}$ we have
$$
\R=\A\cdot\B
$$
with
$$
\A,\B\subset\N,\ \ \ |\A|,|\B|\geq 2.
$$
Since both $H_k$ and $H_{\leq k}$ are infinite, so is $\R$, whence at least one of $\A$ and $\B$, say $\B$ is infinite.

We prove that
\begin{equation}
\label{k4}
\A=\{a_0t:t\in T\}
\end{equation}
with some positive integer $a_0$ and $T\subset\{\Gamma\}$, such that $|T|\geq 2$. Indeed, take distinct elements $a_1,a_2\in\A$. Then for all sufficiently large $b\in\B$ we have
\begin{equation}
\label{ax}
r_1:=a_1b=u_1+\dots+u_\ell
\end{equation}
and
\begin{equation}
\label{bx}
r_2:=a_2b=v_1+\dots+v_h
\end{equation}
with some $r_1,r_2\in\R$, $\ell,h\leq k$, and with $u_1,\dots,u_\ell,v_1,\dots,v_h\in\{\Gamma\}$ such that
\begin{equation}
\label{coprime}
\gcd(u_{i_1},u_{i_2})=\gcd(v_{j_1},v_{j_2})=1\ (1\leq i_1<i_2\leq \ell,1\leq j_1<j_2\leq h).
\end{equation}
We infer from \eqref{ax} and \eqref{bx} that
\begin{equation}
\label{ujj}
a_2(u_1+\dots+u_\ell)-a_1(v_1+\dots+v_h)=0.
\end{equation}
Since there are infinitely many $b\in\B$, and we arrive at \eqref{ujj} whenever $b$ is large enough, this equation has infinitely many solutions $u_1,\dots,u_\ell,v_1,\dots,v_h\in\{\Gamma\}$ with the property \eqref{coprime}. However, by Theorem C this can hold only if, after changing the indices if necessary,
\begin{equation}
\label{k9}
a_2u_1=a_1v_1.
\end{equation}
Let $d_1,d_2$ be the maximal positive divisors of $a_1,a_2$ from $\{\Gamma\}$, respectively. Write
\begin{equation}
\label{k9.5}
a_1=a_1'd_1\ \ \ \text{and}\ \ \ a_2=a_2'd_2,
\end{equation}
and observe that by the pairwise coprimality of the elements of $\Gamma$ both $d_1,d_2$ and $a_1',a_2'$ are uniquely determined. In particular, none of $a_1',a_2'$ is divisible by any element of $\Gamma$. Equations \eqref{k9.5} together with \eqref{k9} imply
$$
a_2'd_2u_1=a_1'd_1v_1,
$$
where $d_2u_1,d_1v_1\in\{\Gamma\}$. We know infer that
$$
a_0:=a_1'=a_2'
$$
and
$$
a_1=a_0t_1,\ \ \ a_2=a_0t_2\ \ \ \text{with}\ t_1,t_2\in\{\Gamma\}.
$$
It is important to note that $a_0$ is the greatest positive divisor of $a_1$ (and of $a_2$) which is not divisible by any element of $\Gamma$. Considering now $a_1,a_i$ instead of $a_1,a_2$ for any other $a_i\in \A$, we get in the same way that
$$
a_i=a_0t_i\ \ \ \text{with}\ t_i\in\{\Gamma\}.
$$
This proves \eqref{k4}.

Write $\Gamma=\{n_1,\dots,n_s\}$ and put $m:=\min(s,k)$. Denote by $\R^\circ$ the subset of $\R$ consisting of sums $u_1+\dots+u_k$ with $u_1,\dots,u_m\in\{\Gamma\}\setminus\el$ such that
\begin{equation}
\label{k10}
u_i=
\begin{cases}
n_i^{\alpha_i}&\text{with}\ \alpha_i>1\ \text{for}\ i\leq m,\\
1&\text{for}\ s<i\leq k\ (\text{if}\ s<k).
\end{cases}
\end{equation}
Clearly, $\R^\circ$ is an infinite set. Take $r_1\in\R^\circ$ of the form
$$
r_1=u_1+\dots+u_k
$$
with $u_1,\dots,u_k$ satisfying \eqref{k10}. By what we have already proved, we can write
$$
r_1=a_0t_1b
$$
with some $t_1\in T$ and $b\in\B$. Put $r_2=a_0t_2b$ with some $t_2\in T$, $t_2\neq t_1$ such that $r_2\in\R$. Writing
$$
r_2=v_1+\dots+v_h
$$
with pairwise coprime $v_1,\dots,v_h\in\{\Gamma\}$, we get
\begin{equation}
\label{k11}
t_2(u_1+\dots+u_k)-t_1(v_1+\dots+v_h)=0.
\end{equation}
Recall that by assumption, $u_i\in\{\Gamma\}\setminus\el$ for $i=1,\dots,m$. Hence we must have $h\geq m$, and repeatedly applying Corollary \ref{cork} (after changing the indices if necessary) we get
$$
t_2{u_i}-t_1{v_i}=0\ \ \ (i=1,\dots,m)
$$
whence
$$
\frac{u_1}{v_1}=\dots=\frac{u_m}{v_m},
$$
that is
$$
u_1v_i=v_1u_i\ \ \ (2\leq i\leq m).
$$
If $m>1$, then this by the coprimality of $u_1,\dots,u_k$ and $v_1,\dots,v_k$ gives $u_i=v_i$ $(i=1,\dots,m)$. This is a contradiction, which proves the theorem whenever $m>1$.

So we are left with the only possibility $m=1$, that is, $s=1$. Then, letting $\Gamma=\{n\}$, equation \eqref{k11} reduces to
\begin{equation}
\label{eqc}
t_2n^{\alpha_1}-t_1n^{\alpha_2}=c,
\end{equation}
where $c=t_1w-t_2(k-1)$ with some $0\leq w\leq k-1$. For any fixed $c\neq 0$ the above equation has only finitely many solutions in non-negative integers $\alpha_1,\alpha_2$. Indeed, we may easily bound $\min(\alpha_1,\alpha_2)$ first, and then also $\max(\alpha_1,\alpha_2)$.
Hence we may assume that $c=0$ in \eqref{eqc}. Observe, that in the case of the set $H_k$ we have $w=k-1$, whence we get $t_1=t_2$, a contradiction.

So in what follows, we may assume that we deal with the set $H_{\leq k}$. Observe that for any large $\beta$, both $n^\beta$ and $n^\beta+1$ belong to $\R$. Hence, in view of \eqref{k4} we get $a_0=1$, and all elements of $\A$ are powers of $n$. This implies that $1\in\A$: indeed, since all elements of $\A$ are powers of $n$, we can have $n^\beta+1\in\R$ only if $1\in\A$ (and $n^\beta+1\in\B$). Recall that $|\A|\geq 2$; let $n^\alpha\in\A$ with some $\alpha>0$, and assume that $\alpha$ is minimal with this property. Obviously, for all large $\beta$ we must have $n^\beta+i\in\B$, for all $0\leq i<k$. One of $k-2,k-1$ is not divisible by $n$; write $j$ for this number.
(Note that for $k=2$ we have $j=1$.) Then, for all large $\beta$, we must have $n^\beta+j\in\B$. Consequently, we have
$$
n^{\alpha+\beta}+n^\alpha j\in\R.
$$
However, this implies that
$$
n^\alpha j\leq k-1.
$$
Hence, in view of $j\in\{k-2,k-1\}$ (with $j=1$ for $k=2$) we easily get that the only possibility is given by
$$
n=2,\ \ \ \ \ \alpha=1,\ \ \ \ \ k=3.
$$
Thus the theorem follows.
$\qed$

\end{document}